\documentclass[12pt]{amsart}
\usepackage{amssymb}
\usepackage{stmaryrd}
\usepackage{enumerate}
\usepackage{yfonts}
\usepackage{titletoc}
\usepackage{hyperref}
\usepackage{amsmath}
\usepackage{stmaryrd}
\usepackage{amsmath}
\usepackage{amsthm}
\usepackage{amssymb}

\theoremstyle{plain}
\newtheorem{theorem}{Theorem}

\newtheorem{lemma}[theorem]{Lemma}
\newtheorem{proposition}[theorem]{Proposition}

\theoremstyle{definition}
\newtheorem{example}{Example}
\newtheorem{definition}[theorem]{Definition}

\newcommand{\PP}{\mathbb P}
 
 \newcommand{\R}{\mathbb R}
 
 \newcommand{\T}{\mathbb T}
\newcommand{\Z}{\mathbb Z}

\newcommand{\cD}{\mathcal D}
\newcommand{\cF}{\mathcal F}

\newcommand{\cP}{\mathcal P}

 \newcommand{\bz}{\mathbf z}

\newcommand{\Hb}{\bar{H}}

\newcommand{\be}{\beta}

 \newcommand{\de}{\delta}
\newcommand{\fui}{\varphi}
\newcommand{\ep}{\varepsilon}

 \newcommand{\Om}{\Omega}
 \newcommand{\te}{\theta}
 
\newcommand{\De}{\Delta}

\newcommand{\diver}{\operatorname{div}}

\newcommand{\sop}{\operatorname{supp}}

\begin{document}

\begin{abstract}
In \cite{G} Diogo Gomes developed techniques and tools with the purpose of 
extending the Aubry-Mather theory in a stochastic setting, namely he proved the 
existence of stochastic Mather measures and their properties. These results were also
 extended in the time-dependent setting in the doctoral thesis the author
 \cite{GV}.
However to construct analogs to the Aubry–Mather measures for nonconvex 
Hamiltonians it is necessary to use the adjoint method introduced by Evans \cite{E} 
and H. V. Tran  \cite{T}, the construction of the measures is in \cite{CGT}. The 
main goal of this paper 
is to construct Mather measures for space-time periodical nonconvex Hamiltonians 
using the techniques in \cite{E}, \cite{T} and \cite{CGT} .
\end{abstract}
\title [MATHER Measures  for space-time periodical nonconvex Hamiltonians]
{MATHER measures  FOR space-time periodical nonconvex Hamiltonians}
\author[Eddaly Guerra-Velasco]{Eddaly Guerra-
Velasco}
\address{CONACYT-Universidad Autónoma de Chiapas, Facultad de Ciencias en F\'isica y Matem\'aticas, M\'exico.}
\email{eddalyg@yahoo.com.mx, eguerra@conacyt.mx}
\maketitle

\section{Introduction.}

\subsection{The Convex Case.}

Let $\T^{d+1}$ be the $d+1$ - torus and consider a smooth periodic Tonelli Hamiltonian 
$H:\T^{d+1}\times \R^d\to\R$. 
 

Let $L:\T^{d+1}\times\R^d\to\R$ be the Lagrangian associated to the 
Hamiltonian: 
\begin{equation}\label{lagrangian}
 L(x,v,t)=\max_p\, pv-H(x,p,t), 
\end{equation}  
 for every $(x,v,t) \in \T^{d+1}\times\R^{d}$.
 
Now we consider the corresponding flow of the time dependent Hamiltonian:
\begin{align}
&\dot x=D_p H(x,p,t),\\\notag
&\dot p=-D_x  H(x,p,t).
\end{align}
Now the dynamics transforms to 
 \begin{align}
&\dot X=D\Hb(P),\\\notag
&\dot P=0.
\end{align}
under the change of variables
\[
(p,x)\to(P,X)
\]
where $p=P+D_xu(x,P,t)$, $X=x+D_Pu(x,P,t)$ if we suppose that both $u(x,P,t)$ 
and
$\Hb(P)$ are smooth functions and $u(x,P,t)$ satisfies the time dependent Hamilton-Jacobi equation
\begin{equation}\label{HJ}
u_t+H(x,D_xu,t)=\Hb(P).
\end{equation}
\begin{definition}
A continuous function $u:\T^{d+1}\to\R$ is called a 
{\em forward viscosity solution} of \eqref{HJ} if it satisfies the two properties.
\begin{enumerate}[(1)]
\item If $v$ is a $C^1$ function and $u-v$ has a local maximum at $(x,t)$, then
$$
v_t+H(x,D_xv(x,t),t)\geq\Hb(P) ,
$$
\item If $v$ is a $C^1$ function and $u-v$ has a local minimum at $(x,t)$, then
$$
v_t+H(x,D_xv(x,t),t)\leq\Hb(P) .
$$
\end{enumerate}

{\em Backward viscosity solutions} are defined by reversing both inequalities.
\end{definition}
It is known (\cite{CIS}, \cite{EG}) that there is only one value $\Hb(P)$, such that \eqref{HJ}
 has a time periodic viscosity solution.

As a consequence of the semilinearity and convexity there is a consequence map
$\Phi:\T^{d+1}\times\R^d\to\T^{d+1}\times\R^d,$ given by 
$\Phi(x,v,t)=(x, D_vL(x,v,t),t),$
wich is well defined and one-to-one.

Recall the Poisson bracket,
\[
\{F\,,G\}:= D_pF\cdot D_xG-D_xF\cdot D_pG. 
\] 

In Hamiltonian coordinates, the property of invariance for a probability measure $\nu$ can be written as
\[
\int_{\T^{d+1}\times\R^{d}} \phi_s+ \{H,\phi \}\,d\mu=0
\]
for every $\phi\in C_c^1 (\T^{d+1}\times\R^{d})$, where $\mu=\Phi_{\#}\nu$
is the push-forward of the measure $\nu$ with respect to the map $\Phi$, i.e. the 
measure $\mu$ such that 
\[
\int_{\T^{d+1}\times\R^{d}}\phi(x,p,t)d\mu(x,p,t)=\int_{\T^{d+1}\times\R^{d}}
\phi(x, D_vL(x,v,t),t)d\nu(x,v,t)\]
for every $\phi\in C_c^1(\T^{d+1}\times\R^{d})$.

Denoting by $\cP(\T^{d+1}\times \R^d)$  the class of probability measures on 
$\T^{d+1}\times \R^d$, and taking $\Om=\T^{d+1}\times\R^{d}$, where 
$(x,v,t)=z$ represents a generic point
$z\in\Om$ with $(x,t)\in\T^{d+1}$ and $v\in\R^{d}$. 

Now let $\cD$ be the class of probability measures in $\Om$ that are invariant under 
the Euler-Lagrange flow, so we have
\[
\cD=\Big\{\nu\in\cP(\Om):\int_\Om \phi_t +\{H, \phi \}d\Phi_{\#}\nu(x,p,t)=0 \text{ 
for every } \phi\in C_c^1(\Om)\Big\},
\]

and the set of holonomic measures
\[
\cF=\{\nu\in\cP(\Om): \int_\Om \psi_t+v\cdot\,D\psi(x,t)\,d\nu(x,v,t)=0,\, \text{ for 
every } 
\psi\in C^1(\T^{d+1})\}.
\]
We recall the Mather problem
\begin{equation}\label{M1}
\min_{v\in\cF}\int _\Om L(x,v,t)\,d\nu(x,v,t),
\end{equation}
a more general version of \eqref{M1} consists in studying for each $P\in\R^d$ fixed
\begin{equation}\label{M2}
\min_{\nu\in\cF}\int_\Om(L(x,v,t)-P\cdot v)d\nu.
\end{equation}

Any minimizer of \eqref{M2} is a Mather measure, now the following proposition will be helpful to prove an important result.


\begin{proposition}
Let $H:\T^{d+1}\times\R^d\to\R$ be a smooth function that satisfies the classical
hypotheses.
 Let $P\in\R^d$, $\nu\in\cP(\Om)$ be a minimizer of \eqref{M2} and set
 $\mu=\Phi_{\#}\nu$.
 Then 
 \begin{enumerate}[(i)]
 \item $\mu$ is invariant under the Hamiltonian dynamics, i.e.,
 $$
 \int_\Om \phi_t +\{H, \phi\}d\mu(x,p,t)=0,\, \text{\,\, for every  } \phi \in 
 C^1_c(\Om)
 $$
 \item $\mu$ is supported on the graph\label{sigma}
 $$
 \Sigma:=\{(x,p,t)\in\T^{d+1}\times\R^d: p=P+D_xu(x,P,t)\}
 $$
 where $u$ is any viscosity solution of \eqref{HJ}.
\end{enumerate}
\end{proposition}

The proof of the proposition is a consequence of results in \cite{B} and \cite{CIS}.

As in \cite{CGT} the following theorem gives a characterization of Mather measures 
in the time dependent convex case.

\begin{theorem}\label{Mm}
Assume $H:\T^{d+1}\times\R^d\to\R$ is a smooth function that satisfies the  
classical hypotheses of convexity, superlinearity, and periodicity and let $P\in\R$.
Then $\nu \in\cP(\Om)$ is a solution of
\[
-\min_{\nu\in\cF}\int_\Om(L(x,v,t)-P\cdot v)d\nu(x,v,t),
\]
if and only if
\begin{enumerate}[(a)]
\item$\int_\Om\,\phi_t + H(x,p,t)d\mu=\Hb(P)=H(x,p,t)\,$    $\,\mu$ a.e.,\label{a}
\item$\int_\Om\phi_t+(p+P)\cdot\,  D_pH(x,p,t)\,d\mu(x,p,t)=0$,\label{b}
\item$\int_\Om \phi_t+D_pH(x,p,t)\cdot D\phi(x,p,t)=0$\label{c}, \text{for every} 
$\phi\in C^1(\T^{d+1})$.
\end{enumerate} 
where $\mu=\Phi_{\#}\nu$ and $\Hb(P)$ is the 
unique value such that \eqref{HJ} has a time periodic viscosity solution.
\end{theorem}
\begin{proof}

To simplify, we will assume $P=0$. Let us prove that $\mu=\Phi_{\#}\nu$ satisfies 
(a)-(c). 
From \eqref{sigma} of the last proposition, and \eqref{lagrangian}, we have that
\[
\int_{\Om}\phi_t + H(x,p,t)d\mu=\Hb(0),
\]
so \eqref{a} holds.

Now, we know that
\[
H(x,p,t)=p\cdot D_pH(x,p,t)-L(x, D_pH(x,p,t),t),
\]

 and from 
\eqref{a} it follows that 
\[
\int_{\Om}\phi_t + p\cdot D_pH(x,p,t)\,d\mu=0.
\]
Finally \eqref{c} follows from that $\nu\in\cF$.

Reciprocally let $\mu\in\cP({\Om})$ such that (a), (b) and (c) holds, and we will 
show that 
$\nu=\Phi_{\#}\mu$ is 
a 
minimizer of \eqref{M2}.

Now observe  that $\nu\in\cF$, then 
\[
\int_{\Om}\psi_t+v\cdot D\psi(x,t)\,d\nu=\int_{\Om}\psi_t+D_pH(x,p,t)\cdot 
D\psi(x,t)=0,
\]
for every $\psi\in C^1(\T^{d+1})$.

The fact that $\nu$ is a minimizer is obtained by using \eqref{a} and \eqref{b}
\begin{equation*}
\Hb(0)=\int_{\Om}\phi_t+H(x,p,t)d\mu=\int_{\Om}\phi_t+p\cdot D_pH-
Ld\mu=\int_{\Om}-Ld\mu.
\end{equation*}
\end{proof}
The previous characterization will help us to define Mather measures in the 
nonconvex case.
\subsection{The Nonconvex Case}
Throughout the paper, we will assume that

\begin{enumerate}[i.]\label{hipo}
\item $H$ is smooth,
\item $H(\cdot,p,t)$ is $\Z^{d+1}$-periodic for $(p,t)\in\R^{d+1}$,
\item\label{iii} There exists a continuous function $\chi:[0,+\infty)\to\R$ such that
\begin{equation}\label{crecimiento}
\int_0^\infty \chi(u)^{-1}du=\infty \text{ and } |H(x,p,t)|\leq\chi(|p|).
\end{equation}
\end{enumerate}
\begin{example}
Consider
$$
H(x,p,t)=|p_1|-|p_2|+V(x,t)
$$
If we take $\chi(u)=2u+C$, 
$$
\int_0^\infty \chi(u)^{-1}du =\infty
$$
and
$|H(x,p,t)|=||p_1|-|p_2|+V(x,t)|\leq \chi(|p|)$.

\end{example}
We extend the definition of Mather measure in the nonconvex and time dependent  
setting:
\begin{definition}\label{Mathermea}
We say that a measure $\mu\in\cP(\Om)$ is a 
Mather measure if there exists 
$P\in\R^d$ such that properties \eqref{a}-\eqref{c} in Theorem \ref{Mm} are 
satisfied.
\end{definition}

Our main result is:

\begin{theorem}\label{main}Assume that the Hamiltonian is a smooth function that 
satisfies the 
conditions (i.)-(iii.) and let $\{\mu^\ep\}_{\ep>0}$ be the family of measures 
defined in \eqref{mue}. Then there exist a Mather measure $\mu$ and a  
nonnegative and symmetric $d\times d$  matrix $(m_{kj})_{k,j=1,\dots d}$ of 
 Borel 
measures called the \it{dissipation measure}, such that:
\begin{enumerate}[(1)]
\item\label{mucom} $\mu^\ep\overset{*}\rightharpoonup\mu$ in the sense of 
measures up to subsequences,
\item\label{idenmat} $\int_{\Om}\fui_t+\{H,
\fui\}d\mu+\int_{\Om}\fui_{p_kp_j}dm_{kj}=0
\text{ for all } \fui\in C_c(\Om),$
\item\label{sops} $\sop\mu$ and $\sop m$ are compact. 
\end{enumerate}
\end{theorem}

\section{Uniform Derivate Bounds}
Let us consider the equation:
\begin{equation}\label{HJV}
\phi_t^\ep+\ep\Delta\phi^\ep+H(x,P+D\phi^\ep(x,t),t)=\Hb(P).
\end{equation}

\begin{lemma}\label{lip-unif}
 The periodic solutions of \eqref{HJV} have  first derivatives, uniformly bounded
in $\ep$.
\end{lemma}

\begin{proof}[Sketch of the proof.]
For every $\ep>0$ let us consider the following problem
\begin{equation}\label{aux}
\phi_t^\ep+\Delta\phi^\ep+H(x,D\phi^\ep,t)-\ep\phi^\ep=0
\end{equation}

The above equation has a unique smooth solution $\phi^\ep$ in $\R^{d+1}$ which is $\Z^{d+1}$ periodic \cite{GV}.
First,  we proved that $D\phi^\ep$ is uniformly bounded, by following \cite{bar} we proved that there exists $K>0$ depending only on $H$ such that

\begin{equation}
    \sup\|D\phi^\ep(\cdot,t)\|_\infty\leq K
\end{equation}
 Finally if we take 
$g=d_t^2phi^\ep+|D\phi^\ep|^2$and using the Bernstein's method we prove that
$d_t\phi^\ep$ is uniformly bounded.

\end{proof}
\begin{theorem}
For every $\ep>0$ and every $P\in\R^d$, there exists a unique number 
$H^\ep(P)\in\R$ such that the equation \eqref{HJV} admits a unique (up to 
constants) $\Z^{d+1}$ periodic viscosity solution. Moreover, for every $P\in\R^d$
$\underset{\ep\to 0^+}\lim \Hb^\ep(P)\to\Hb(P)$ and $\phi^{\ep}\to\phi_0$ 
uniformly (up to subsequences), where $\phi_0:\T^{d+1}\to\R$ is that \eqref{HJ} 
is satisfied in the viscosity sense.
\end{theorem}
\begin{proof}
 
The theorem follows by Lemma \ref{lip-unif}, the stability theorem 
for viscosity solutions and the Arzela-Ascoli Theorem.
\end{proof}

\section{Stochastic Measures}

\begin{definition}
Let $\ep>0$ and $P\in\R^d$. The linearized operator associated to \eqref{HJV} is defined as
$L_{\ep,P}:C^2(\T^{d+1})\to C(\T^{d+1}):$
\begin{equation}\label{Lep}
L_{\ep,P}\,\psi=\psi_t+\ep\Delta\psi+D_pH(x,P+D\phi^\ep(x,t),t)D\psi,
\end{equation}
for every $\psi\in C^2(\T^{d+1}).$
\end{definition}

As in \cite{CGT}, we denote by $\be$ either a direction in $\R^d$ (i.e., 
$\be\in\R^d$
with $|\be|=1$) or a parameter (for example $\be= P_i$ for some $i \in\{1,\dots, 
d\}$).
When $\be=P_i$ for some $i\in\{1,\dots, d\}$ the symbols $H_\be$ and 
$H_{\be\be}$ have to be understood as $H_{p_i}$ and $H_{p_ip_i}$ respectively.
If we derive \eqref{HJV} with respect to $\be$ and recalling \eqref{Lep}
we get
$$
 L_{\ep,P}\,\phi_\be^\ep=\phi_{t\be}^\ep+\ep\Delta\phi_\be^\ep+
 D_pH(x,P+D\phi^\ep(x,t),t)D\phi_\be^\ep+H_\be=\Hb^\ep,
$$
so
\begin{equation}\label{ELe,P}
L_{\ep,P}\,\phi_\be^\ep=\Hb^\ep-H_\be.
\end{equation}
As before, let $\Om=\T^{d+1}\times\R^{d}$, where $(x,v,t)$ represents a generic 
point with $(x,t)\in\T^{d+1}$ and $v\in\R^{d}$.  We need to introduce a
probability space $(\Om,{\mathcal B},\PP)$ endowed with a Brownian
motion $W(t):\Omega\to \T^d$ on the flat $d$-torus. Let $\ep>0$, to simplify we
set $P=0$ and we introduce the time dependent vector field \cite{Fl},
$U_\ep(x,t)=D_pH(x,D\phi_\ep(x,t),t)$ and consider the solution $X_\ep(s)$
of the stochastic differential equation
\begin{equation}  \label{eq:optistoc}
\begin{cases}
  dX^\ep(s) &=U_\ep(X_\ep(s),s)ds+\sqrt{2\ep}\,dW(s),\\
X^\ep(t)&=x.
 \end{cases}
\end{equation}
And the momentum variable is defined as
\[
p^\ep(t)=D\phi^\ep(X^\ep(t),t).
\]

Now suppose $\bz:[0,+\infty)\to\R^d$ is a solution to the stochastic differential 
equation 
\[
d\bz_i=a_i\,ds+b_{ij}dW(s)
\]
with $a_i$ and $b_{ij}$ bounded and progressively measurable processes. Let
$\fui:\R^d\times \R\to\R$ be a smooth function where $\fui(\bz,t)$ satisfies the 
It\^{o} formula:
\begin{equation}\label{Ito}
d\fui=\fui_{z_i}dz_i+(\fui_t+\frac{1}{2}b_{ij}b_{jk}\,\fui_{z_iz_k})dt.
\end{equation}
From hereafter, we will use Einstein's convention for repeated indices in a sum. Here,
we have $a_i=D_p H(x,D\phi^\ep(x,t),t)$ and $b_{ij}=\sqrt{2\ep}\de_{ij}$.

Therefore, from \eqref{eq:optistoc}, \eqref{Ito} and \eqref{ELe,P},
\begin{align*}
dp_i &=\phi^\ep_{x_ix_j}[D_pH(X^\ep,D\phi^\ep(x,t),t)dt+\sqrt{2\ep}dW^j]+
\phi_{x_it}^\ep(X^\ep(x),t)+\ep\Delta\phi_{x_i}^\ep\,dt,  \\
 &= L_{\ep, P}\,\phi_{x_i}^\ep\, dt+\sqrt{2\ep}\,\phi_{x_ix_j}dW^j\\
 &=H_{x_i}\,dt+\sqrt{2\ep}\,\phi_{x_ix_j}dW^j.
\end{align*}
Thus $(X_\ep,p_\ep)$ satisfies the following stochastic version of the Hamiltonian 
dynamics
\begin{equation}  \label{eq:optistoc1}
\begin{cases}
  dX^\ep(s) &=U_\ep(X_\ep(s),s)ds+\sqrt{2\ep}\,dW(s),\\
dp^{\ep}(s)&=-DH(X^\ep, p^\ep,s)ds+\sqrt{2\ep}\,D^2\phi^\ep dW(s).
 \end{cases}
\end{equation}

Now we are going to study the solution $\phi^\ep$ of \eqref{HJV} along the 
trajectory $X_\ep(s)$. Due to the It\^{o} formula, and the equations \eqref{HJV} and  
\eqref{eq:optistoc}.
\begin{align*}
d\phi^\ep(X_\ep(s))&=D\phi^\ep dX_\ep+(\phi_t^\ep+\ep\Delta\phi^\ep)ds,\\
&=D\phi^\ep[D_pH(x,D\phi^\ep(x,s),s)ds+\sqrt{2\ep}dW(s)]+(\phi_t^\ep+\ep\Delta
\phi^\ep)ds,\\ 
&=(\phi_t^\ep+\ep\Delta\phi^\ep+D_pH(x,D\phi^\ep(x,s),s)\cdot D\phi^\ep)ds                          
+\sqrt{2\ep}D\phi^\ep dW(s),\\
&= L_{\ep, P}\phi^\ep\,ds+\sqrt{2\ep}\,dW(s)=(\Hb_\ep(P)-H+D_p H\cdot 
D\phi^\ep)ds+\sqrt{2\ep}dW(s).
\end{align*}

And using the Dynkin formula, we obtain
\begin{align*}
E[\phi^\ep(X_\ep(T))-\phi^\ep(X_\ep(0))]&=E\Big[\int_0^T(D_pH(X_\ep,D\phi^\ep
(x,t),t)D\phi^\ep+\ep\Delta\phi^\ep)\Big]dt\\
&=E\Big[\int_0^T(D_pH(X_\ep,D\phi^\ep(x,t),t)D\phi^\ep+\Hb_\ep(P)-
H-\phi_t^\ep)dt\Big]
\end{align*}

Now we will associate to each trajectory $(X_\ep,p_\ep,t)$ of \eqref{eq:optistoc1} a 
probability measure $\mu_\ep\in\cP(\T^{d+1}\times\R^d)$ defined by
\begin{equation}\label{mue}
\int_{\T^{d+1}\times\R^d}\phi(x,p,t)d\mu_\ep(x,p,t):=\lim_{T\to\infty}\frac{1}
{T}
E\Big[\int_0^T \phi(X_\ep(t), p^\ep(t),t)\,dt\Big]
\end{equation}
for every $\phi\in C_c(\T^{d+1}\times\R^d).$
Here, the definition makes sense provided the limit is taken over an appropiate
 subsequence. Then using Dynkin's formula, we have that for every $\fui\in
C_c(\T^{d+1}\times\R^d),$
\begin{align}\label{din1}
&E\Big[\fui(X^\ep(T), p^\ep(T), T)-\fui(X^\ep(0), p^\ep(0), 
0)\Big]\\ \notag
&=E\Big[\int_0^T \fui_t+(D_x\fui\cdot D_pH-D_p\fui\cdot D_x H) dt \Big]\\
&+E\Big[\int_0^T(\ep\fui_{x_ix_i}+2\ep\,
\phi^\ep_{x_ix_j}\fui_{x_ip_j}+\ep\phi^\ep_{x_ix_k}\phi^\ep_{x_ix_j}
\fui_{p_kp_j})\,dt\Big].\notag
\end{align}

Dividing the equation \eqref{din1} by $T$ and taking the limit when $T\to\infty$
along a suitable subsequence we obtain:
\begin{equation}\label{din2}
\int_{\T^{d+1}\times\R^d}
\fui_t+\{H,\fui\} d\mu^\ep+\int_{\T^{d+1}\times\R^d}(\ep\fui_{x_ix_i}+2\ep\,
\phi^\ep_{x_ix_j}\fui_{x_ip_j}+\ep\phi^\ep_{x_ix_k}\phi^\ep_{x_ix_j}
\fui_{p_kp_j})\,d\mu^\ep
\end{equation}
Let us define the projected measure $\te_{\mu^\ep}\in\cP(\T^{d+1})$ as follows
\[
\int_{\T^{d+1}}\fui(x,t)d\te_{\mu^\ep}(x,t):=\int_{\T^{d+1}\times\R^d}
\fui(x,t)d\mu^\ep(x,p,t)
\]
for all $\fui\in C(\T^{d+1})$. And using test functions that do not depend on $p$ in 
the last definition:
\begin{equation}\label{test}
-\int_{\T^{d+1}}D_pH\cdot D_x\fui \,d\te_{\mu^\ep}=\int_{\T^{d+1}}
(\fui_t+\ep\De\fui)d\te_{\mu^\ep}
\end{equation}
for all $\fui\in C^2(\T^{d+1})$.

Given $\phi^\ep$, let us consider the partial differential equation
\[
\te_t^\ep-\ep\De\te^\ep+\diver(\te^\ep\cdot D_p H(x,D\phi^\ep,t))=0,
\]
From lemma 32 and lemma 33 in \cite{GV}, we have that 0 is the principal 
value of Fokker-Planck operator
$$
N(\te^\ep)=-\te^\ep_t+\Delta\te^\ep-\diver(\te^\ep\cdot D_p H(x,D\phi^\ep,t))
$$
and so $\mu^\ep$ can be defined as a unique measure such that
$$
\int_{\T^{d+1}\times\R^d}\psi(x,p,t)d\mu^\ep(x,p,t)=\int_{\T^{d+1}}
\psi(x,D\phi^\ep(x,t),t)d\te^\ep(x,t), 
$$
for every $\psi\in C_c(\T^{d+1}\times\R^d)$.

\subsection{Uniform Estimates.}

\begin{lemma}
We have the following estimates:

\begin{equation}\label{1}
 \ep\int_{\T^{d+1}}|D\phi^\ep_\be|^2d\te_{\mu^\ep}=2\int_{\T^{d+1}}
 \phi_\be^\ep(H_\be-\Hb^\ep_\be)d\te_{\mu^\ep}=-2\int_{\T^{d+1}}
 \phi_\be^\ep (\Hb^\ep_\be-H_\be)d\te_{\mu^\ep},
\end{equation}
\begin{equation}\label{2}
\int_{\T^{d+1}}(\Hb_\be^\ep-H_{\be\be}-2D_pH_\be\cdot D\phi_\be^\ep-
D^2_{pp}HD\phi_\be^\ep\cdot D\phi_\be^\ep)d\te_{\mu^\ep}=0,
\end{equation}
\begin{equation}\label{3}
\ep\int_{\T^{d+1}}|D\phi^\ep_{\be\be}|^2\,d\te_{\mu^\ep}=
-2\int_{\T^{d+1}}\phi^\ep_{\be\be}(\Hb_\be^\ep-H_{\be\be}-2D_pH_\be\cdot 
D\phi_\be^\ep-D^2_{pp}H:D\phi_\be^\ep\otimes 
D\phi_\be^\ep)\,d\te_{\mu^\ep}.
\end{equation}
\end{lemma}
\begin{proof}
Recalling \eqref{Lep}, we obtain
\begin{align*}
L_{\ep,P}|\phi_\be^\ep|^2&=\frac{\partial}{\partial t}<\phi_\be^\ep,
\phi_\be^\ep >+\ep\De<\phi_\be^\ep,\phi_\be^\ep 
>+D_pH(x,P+D<\phi_\be^\ep,
\phi_\be^\ep >,t)+H_\be=\Hb_\be^\ep\\
&=2\phi_{t\be}^\ep\cdot \phi_\be^\ep+2\ep 
D\phi_\be^\ep\cdot\phi_\be^\ep+2\ep\De \phi_\be^\ep\cdot  \phi_\be^\ep+
D_pH(x,P+D\phi^\ep(x,t),t)\cdot 2D\phi_\be^\ep\cdot \phi_\be^\ep.
\end{align*}
Thus
\[ 
L_{\ep,P}|\phi_\be^\ep|^2=2L_{\ep,P}\,\phi_\be^\ep\cdot  \phi_\be^\ep
+\ep|D \phi_\be^\ep|^2=2 \phi_\be^\ep(\Hb_\be^\ep-H_\be)+
\ep|D\phi_\be^\ep|^2.
\]
Integrating with respect to $\te_{\mu^\ep}$ and using \eqref{test}, we get 
\eqref{1}.

To obtain \eqref{2} we differentiate \eqref{ELe,P} with respect to
$\be$, we have:
\begin{equation}\label{Lepbb}
L_{\ep,P}\,\phi^\ep_{\be\be}=\Hb_\be^\ep-H_{\be\be}-2D_pH_\be\cdot 
D\phi_\be^\ep-D^2_{pp}H:D\phi_\be^\ep\otimes D\phi_\be^\ep,
\end{equation}
Integrating again with respect to $\te_{\mu^\ep}$ and using \eqref{test} we get 
\eqref{2}.

On the other hand 
\[
L_{\ep,P}|\phi^\ep_{\be\be}|^2=2L_{\ep,P}\phi^\ep_{\be\be}\cdot 
\phi^\ep_{\be\be}+\ep|D\phi^\ep_{\be\be}|^2, 
\]

using \eqref{Lepbb} we obtain
\[
\frac{1}{2}L_{\ep,P}|\phi^\ep_{\be\be}|^2=\phi^\ep_{\be\be}(\Hb_\be^\ep-
H_{\be\be}-2D_pH_\be\cdot D\phi_\be^\ep-D^2_{pp}H:D\phi_\be^\ep\otimes 
D\phi_\be^\ep)+\frac{\ep}{2}|D\phi^\ep_{\be\be}|^2,
\]
once again, integrating with respect to $\te_{\mu^\ep}$ and by \eqref{test} we get 
\eqref{3}.
\end{proof}

Following the techniques of \cite{CGT}, \cite{E} and \cite{T}, we will obtain several estimates that will be useful in the future.
\begin{proposition}\label{unifst}
We have the following
\begin{equation}\label{d2phi}
\ep\int_{\T^{d+1}}|D^2_{xx}\phi^\ep|^2d\te_{\mu^\ep}\leq C
\end{equation}
\begin{equation}\label{d2Px}
\ep\int_{\T^{d+1}}|D^2_{Px}\phi^\ep|^2d\te_{\mu^\ep}\leq 
\int_{\T^{d+1}}|D_P\phi^\ep|^2d\te_{\mu^\ep}+\int_{\T^{d+1}}|D_P 
H-D_P\Hb^\ep|^2d\te_{\mu^\ep}.
\end{equation}
\begin{equation}\label{dxphixx}
\ep\int_{\T^{d+1}}|D\phi_{x_ix_i}^\ep|^2\leq C(1+\int_{\T^{d+1}}|
D^2_{xx}\phi^\ep|^3 d\te_{\mu^\ep})
\end{equation}
\end{proposition}

\begin{proof}
Taking $\be=x_1,x_2,\dots,x_{d+1}$ respectively in \eqref{1}, we have
\[
\ep\int_{\T^{d+1}}|D\phi^\ep_{x_i}|^2d\te_{\mu^\ep}=2\int_{\T^{d+1}}
\phi^\ep_{x_i}H_{x_i}d\te_{\mu^\ep},
\]
and adding these $d+1$ identities we obtain
\[
\ep\int_{\T^{d+1}}|D^2_{xx}\phi^\ep|^2d\te_{\mu^\ep}=2\int_{\T^{d+1}}
D_x\phi^\ep\cdot D_x Hd\te_{\mu^\ep},
\]
now, due to Lemma. \ref{lip-unif}, $|D_x\phi^\ep|$, $|D_x H|$ are uniformly 
bounded, thus we get \eqref{d2phi}.

Now, the relation \eqref{d2Px} follows by taking $\be=P_1, P_2,\dots, P_{d+1}$ in 
\eqref{1}, adding the $d+1$ identities
\[
\ep\int_{\T^{d+1}}|D^2_{Px}\phi^\ep|^2d\te_{\mu^\ep}=2\int_{\T^{d+1}}
D_P\phi^\ep[D_P H-D_P\Hb^\ep]d\te_{\mu^\ep}
\]
And using the Young's inequality.

To obtain \eqref{dxphixx}, taking $\be=x_i$ in \eqref{3}
\[
\ep\int_{\T^{d+1}}|D\phi_{x_ix_i}^\ep|^2=2\int_{\T^{d+1}}\phi^\ep_{x_ix_i}
(H_{x_ix_i}+2D_pH_{x_i}\cdot D\phi_{x_i}^\ep-
D^2_{pp}H:D\phi_{x_i}^\ep\otimes D\phi_{x_i}^\ep)d\te_{\mu^\ep}
\]
Due to Lemma \ref{lip-unif}, we have that $|H_{x_i,x_i}|,\, |D_pH_{x_i}|,\, |
D^2_{pp} H|\leq C$ on the support of $d\te_{\mu^\ep}$, so
\begin{align*}
\ep\int_{\T^{d+1}}|D\phi_{x_ix_i}^\ep|^2 &\leq C\Big(\int_{\T^{d+1}}|
D^2_{xx}\phi^\ep|^2d\te_{\mu^\ep }+\int_{\T^{d+1}}|
D^2_{xx}\phi^\ep|^2d\te_{\mu^\ep }+\int_{\T^{d+1}}|
D^2_{xx}\phi^\ep|^3d\te_{\mu^\ep }\Big)\\
&\leq C\Big( 1+\int_{\T^{d+1}}|D^2_{xx}\phi^\ep|^3d\te_{\mu^\ep }\Big).
\end{align*}
\end{proof}
\section{Existence of Mather measures.}
Now we are able to prove the existence of Mather measures.

\begin{proof}[Proof of Theorem \ref{main}.]
The proof straightforward noticing that $\{\phi^\ep\}$ have a uniform Lipschitz
estimate, therefore there exists a compact set  $K\subset\Om$ such that
$\sop\mu^\ep\subset K\,\forall\ep>0$. Moreover, up  to subsequences, we have  
$\mu^\ep\overset{*}\rightharpoonup\mu$, that is
\[
\lim_{\ep\to0}\int_{\Om}\phi\,d\mu^\ep\to\int_{\Om}\phi\, d\mu
\]
for every function $\phi\in C_c(\Om)$, for some probability measure 
$\mu\in\cP(\Om)$,
and also it follows that $\sop\mu\subset K$.

To obtain \eqref{idenmat}, let us remember \eqref{din2} particularly the second 
term
\begin{equation}
\int_{\Om}(\ep\fui_{x_ix_i}+2\ep\phi^\ep_{x_ix_j}\fui_{x_ip_j}+
\ep\phi^\ep_{x_ix_k}\phi^\ep_{x_ix_j}\fui_{p_kp_j})\,d\mu^\ep
\end{equation}
But 
\begin{align}
\Big |\int_{\Om}(\ep\fui_{x_ix_i}+2\ep\phi^\ep_{x_ix_j}\fui_{x_ip_j}\Big|
\,d\mu^\ep&\leq C\ep+C\ep\int_{\Om}|\phi^\ep_{x_ix_j}|\,d\mu^\ep\\
                 &\leq C\ep+C\ep^{\frac{1}{2}}.
 \end{align}
 by using the estimates in Proposition \ref{unifst}, so the 
 $$
 \underset{\ep\to0}\lim 
\int_\Om (\ep\fui_{x_ix_i}+2\ep\phi^\ep_{x_ix_j}\fui_{x_ip_j})\,d\mu^\ep=0
$$
Note that 
$\ep\int_{\Om}(\phi^\ep_{x_ix_k}\phi^\ep_{x_ix_j}\fui_{p_kp_j})\,d\mu^\ep$ 
does not vanish in the limit, through a subsequence for every $k,j=1,\dots, n$ we 
have
$$\ep\int_{\Om}(\phi^\ep_{x_ix_k}\phi^\ep_{x_ix_j}\fui_{p_kp_j})\,d\mu^\ep \to
\int_{\Om}\fui_{p_kp_j}\,dm_{kj},\,\,\forall\fui\in C_c(\Om)$$
for some nonnegative, symmetric $d\times d$ matrix $(m_{kj})\,\,k,j=1,\dots, n$
of Borel measures, so condition \ref{idenmat}, follows.
To obtain \eqref{sops}, recall that $\sup|p^\ep(t)|<\infty$ and the periodicity in 
time.

Now we will prove that $\mu$ satisfies the conditions (a)-(c) 
in Definition \ref{Mathermea} with $P=0$.
Following \cite{CGT}, \cite{E} and \cite{T} consider 
$$
\int_\Om(\phi_t+H(x,p,t)-\Hb^\ep)^2\,d\mu^\ep =\ep^2\int_\Om|
\Delta\phi^\ep(x,t)|^2d\mu^\ep \to 0
$$ 
when $\ep\to 0$ due to \eqref{HJV} and \eqref{d2phi}, thus (a) occurs.

Recalling the equation \eqref{din2} and choosing as a test function 
$\fui=\psi(\phi^\ep(x,t))$
$$
\int_{\Om}\psi'(\phi^\ep)[\phi_t^\ep+D_x\phi^\ep\cdot D_p 
H]d\mu^\ep+\ep\int_\Om\psi'(\phi^\ep)
[\phi_{x_ix_i}^\ep+(\phi_{x_i}^\ep)^2]d\mu^\ep=0,
$$
if $\ep$ goes to zero, we obtain
$\int_{\Om}\psi'(\phi)[\phi_t+p\cdot D_pH]d\mu=0$, choosing $\psi(\phi)=\phi$, 
we obtain b).
Now part c) follows  by choosing in \eqref{idenmat} test functions $\fui$ that do not 
depend on the variable p.
\end{proof}

\end{document}